\documentclass[12pt]{amsart}
\usepackage[margin = 1in]{geometry}
\usepackage{amssymb, amsthm, xcolor, graphicx}
\usepackage{paralist}
\usepackage{amsthm}
\usepackage{young}
\usepackage{bbold}
\usepackage{youngtab}

\newtheorem*{mainthm1}{First Main Theorem}
\newtheorem*{mainthm2}{Second Main Theorem}

\newtheorem{theorem}{Theorem}
\newtheorem{thm}[theorem]{Theorem}
\newtheorem*{theorem*}{Theorem}

\newtheorem{prop}[theorem]{Proposition}

\theoremstyle{definition}
\newtheorem*{definition*}{Definition}
\newtheorem{definition}[theorem]{Definition}

\newtheorem{remark}[theorem]{Remark}
\newtheorem{example}[theorem]{Example}
\newtheorem{exmp}[theorem]{Example}

\newcommand{\bpm}{\begin{pmatrix}}
\newcommand{\epm}{\end{pmatrix}}

\newcommand{\onto}{\twoheadrightarrow}

\newcommand{\spn}{\operatorname{span}}

\newcommand{\aff}[1]{\widetilde{#1}}
\newcommand{\mov}{\operatorname{mov}}

\newcommand{\fixaff}{\operatorname{fix}_{\mathrm{aff}}}
\newcommand{\fixlin}{\operatorname{fix}_{\mathrm{lin}}}

\newcommand{\im}{\operatorname{im}}
\newcommand{\codim}{\operatorname{codim}}

\newcommand{\GL}{\operatorname{GL}}
\newcommand{\GA}{\operatorname{GA}}
\newcommand{\VV}{\widetilde{V}}
\newcommand{\FF}{\mathbf{F}}
\newcommand{\RR}{\mathbf{R}}

\renewcommand{\gg}{\aff{g}}
\newcommand{\rr}{\aff{r}}
\newcommand{\affR}{\aff{R}}

\begin{document}

\title{Reflection length in the general linear and affine groups}
\author{Elise G delMas}
\address{
School of Mathematics \\
University of Minnesota \\
Minneapolis, MN 55455}
\email{delma007@umn.edu}
\author{Joel Brewster Lewis}
\address{
Department of Mathematics \\ 
George Washington University \\ 
Washington, DC 20052}
\email{jblewis@gwu.edu}

\begin{abstract}
We give an intrinsic criterion to tell whether a reflection factorization in the general linear group is reduced, and give a formula for computing reflection length in the general affine group.
\end{abstract}

\maketitle

Given a group $G$ with generating set $R$, the problem of computing the shortest length $\ell_R(g)$ of a tuple $(r_1, \ldots, r_k)$ of elements of $R$ whose product is a given element $g = r_1 \cdots r_k$ of $G$ is in general undecidable.  However, for certain reflection groups $G$, this function is not only easy to compute but in fact carries an intrinsic, geometric meaning.
\begin{theorem}[{{\cite[Lem.~2]{Carter}, 
\cite[Thm.~1]{Scherk},
\cite[Thm.~1]{Dieudonne}}}]
\label{thm:linear}
If $G$ is 
\begin{compactenum}[\quad (a)]
\item a finite real reflection group, or 
\item the real orthogonal group $O_n(\RR)$, or
\item the general linear group $\GL(V)$ of a finite-dimensional vector space $V$, 
\end{compactenum}
and $R$ is the set of elements in $G$ that fix a hyperplane pointwise (the \emph{reflections}), then the reflection length
$\ell_R(g)$ of any element $g$ in $G$ is equal to $\dim \im(g - 1) = \codim \ker(g - 1)$.
\end{theorem}

Here the subspace $\ker(g - 1) =: V^g$ is the \emph{fixed space} of $g$, consisting of all points in $V$ that do not move under the action of $g$, while the complementary space $\im(g - 1)$ is the \emph{moved space} of $g$, consisting of all vectors by which elements of $V$ can be displaced under the action of $g$.

In the cases of a finite real reflection group and the full orthogonal group $O_n(\RR)$, Theorem~\ref{thm:linear} comes with a complementary result, giving a simple rule for determining when a reflection factorization is of minimal length.  Say that a set of lines through the origin is \emph{independent} if the associated direction vectors are linearly independent, and similarly say a set of hyperplanes is independent if the associated linear forms are linearly independent.

\begin{theorem}[{{\cite[Lem.~3]{Carter}, 
\cite[Prop.~2 \& Cor.~1]{BradyWattEuclidean}}}]
\label{thm:linear2}
If $G$ is 
\begin{compactenum}[\quad (a)]
\item a finite real reflection group or 
\item the real orthogonal group $O_n(\RR)$,
\end{compactenum}
then a tuple $(r_1, \ldots, r_k)$ of reflections in $G$ is a shortest reflection factorization of its product (i.e., $\ell_R(r_1 \cdots r_k) = k$) if and only if the set of fixed hyperplanes (equivalently, the set of moved spaces) of the factors is independent.
\end{theorem}

The first main result of this paper is to extend this result to the full general linear group
$\GL(V)$, giving a complete set of results that parallel Theorem~\ref{thm:linear}.  It is perhaps the simplest possible extension to the new context.

\begin{mainthm1}
When $G = \GL(V)$ is the general linear group of a finite-dimensional vector space $V$, 
a tuple $(r_1, \ldots, r_k)$ of reflections in $G$ is a shortest reflection factorization of its product if and only if both the set of reflecting hyperplanes and the set of moved spaces of the factors are independent.
\end{mainthm1}

In fact, we show more: if only one of the two independences holds, then the rank of the dependent collection gives the reflection length of the product.  These results are proved below as Theorem~\ref{thm:elise}.

Recently, there has also been work done extending Theorem~\ref{thm:linear} to the affine setting.  In this context, the group $G$ acts not on a vector space $V$ but on an affine space $\aff{V}$, and is generated by its subset of affine reflections, i.e., the elements that fix an (affine) hyperplane in $\aff{V}$ pointwise.   In this context, the \emph{motion} of a point $x$ in $\aff{V}$ under an element $g$ is the vector $g(x) - x$ in $V$, and the \emph{moved space} $\mov(g)$ is the set of motions of all points in $\aff{V}$ under the action of $g$.  One can show that this moved space is always an affine subspace of $V$, and so has a meaningful notion of dimension.  These definitions lead to the following result.
\begin{theorem*}[{{\cite[Thm.~5.7]{BradyMcCammond}, \cite[Thm.~A]{LMPS}}}]
If $G$ is
\begin{compactenum}[\quad (a)]
\item an affine Coxeter group or 
\item the group of isometries of real Euclidean space,
\end{compactenum}
then the reflection length $\ell_R(g)$ of any element $g$ in $G$ is equal to $\dim(\mov(g)) + d_G(g)$, where $d_G$ is a geometric statistic that varies with the group $G$. 
\end{theorem*}
(A precise version of this statement is given in Section~\ref{sec:affine statements}.)
The second main result of this paper is to extend this result to the general affine group of all affine transformations of a finite dimensional affine space, again giving a complete set of results that parallel Theorem~\ref{thm:linear}.  It is proved below as Theorem~\ref{main affine theorem}.

\begin{mainthm2}
When $G = \GA(\VV)$ is the general affine group of a finite-dimensional affine space $\VV$, the reflection length $\ell_{\affR}(\gg)$ of any element $\gg$ in $G$ is equal to $\dim(\mov(g)) + d_G(\gg)$.  Here $d_G(\gg)$ is equal to $0$, $1$, or $2$, depending on whether $\gg$ is \emph{elliptic}, \emph{parabolic}, or \emph{hyperbolic} (see Definitions~\ref{def:tripartite} and~\ref{def:tripartite for F_2}).
\end{mainthm2}

The proofs of our two main theorems are given in Sections~\ref{sec:linear} and~\ref{sec:affine group}, respectively.  In Section~\ref{sec:remarks}, we include a number of additional remarks.

\section{The general linear group}
\label{sec:linear}
Let $G = \GL(V)$ be the general linear group of a finite-dimensional vector space $V$ over the field $\FF$.  The \emph{fixed space} $V^g$ of an element $g$ in $G$ is the set
\[
V^g := \ker(g - 1) = \{v\in V\mid g(v) = v\}.
\]
An element $s$ in $G$ is a \emph{reflection} if its fixed space has codimension $1$, that is, if $s$ fixes a hyperplane.\footnote{Because we do not require orthogonality, this is not the same as the usual definition over the real numbers.  However, in the context of the real orthogonal group or a finite real reflection group, the reflections under this more general definition are exactly the usual orthogonal reflections.  Bourbaki refers to our reflections as \emph{pseudo-reflections} \cite[Ch.~5, \S2.1]{Bourbaki}, except that they do not require pseudo-reflections to be invertible maps.}   Let $R$ be the set of all reflections in $\GL(V)$.  All of the elementary matrices used in Gaussian elimination are reflections in this sense, so the fact that every invertible matrix is row-reducible to the identity implies that $G$ is generated by $R$.  The \emph{reflection length} of an element $g\in G$ is defined by 
\[
\ell_R(g) := \min\{k \mid g = r_1r_2\cdots r_k \text{ for some } r_i \in R\}.
\]

In the real orthogonal setting, a reflection is determined uniquely by its fixed hyperplane.  Equivalently, it is determined by its \emph{root}, the vector orthogonal to the hyperplane.  If $r$ is an orthogonal reflection with root $\lambda$ of unit length, the action of $r$ on a vector $x \in V$ is given by $r(x) = x - 2\langle \lambda, x\rangle \cdot \lambda$.  The map $x \mapsto -2\langle \lambda, x \rangle$ is a nonzero linear form, i.e., it is a nontrivial linear map from $V$ to the underlying field (in this case, $\RR$).  In general, the collection of linear forms on a vector space $V$ forms the \emph{dual space} $V^*$.  The next result generalizes the formula for orthogonal reflections to describe arbitrary reflections in $\GL(V)$.

\begin{prop}
\label{Vectors}
For any reflection $t$ in $G = \GL(V)$, there exists a nonzero vector $v_t \in V$ and a nonzero linear form $\alpha_t \in V^*$ such that for all $x\in V$, the action of $t$ is given by
\[ 
t(x) = x + \alpha_t(x) \cdot v_t.
\]
Moreover, in this case we have $V^t = \ker(\alpha_t)$ and $\im(t - 1) = \FF \cdot v_t$.  Finally, a pair $(v, \alpha)$ as above defines a reflection if and only if $\alpha(v) \neq -1$.
\end{prop}
\begin{proof}
The first parts of the claim may be found in \cite[Ch.~5, \S2.1]{Bourbaki}; we include their proofs for completeness.

For any reflection $t$ in $G$, we have by definition that $\codim(V^t) = \codim(\ker(t-1)) = 1$.  By the rank--nullity theorem, $\dim(\im(t-1)) = 1$.  Let $v_t$ be a fixed basis vector for the space $\im(t-1)$. Then the action of $t - 1$ can be expressed as $(t - 1)(x) = \alpha_t(x) v_t$ for some nonzero function $\alpha_t: V \to \FF$.  Since $t - 1$ is linear, $\alpha_t$ must be linear, and so $\alpha_t \in V^*$. The first three claims follow immediately from this formula.  

For the final claim, choose a nonzero vector $v$ in $V$ and a nonzero linear form $\alpha$ in $V^*$, and define a linear map $t$ from $V$ to $V$ by the formula $t(x) := x + \alpha(x) \cdot v$.  Since $v \neq 0$, the fixed space $\ker(t - 1)$ is exactly the hyperplane $\ker \alpha$, and so the map $t$ will be a reflection as long as it actually belongs to $\GL(V)$.  The kernel of $t$ consists of those points $x$ such that $x = -\alpha(x) \cdot v$, and so is contained in the one-dimensional space $\FF \cdot v$.  Thus, the kernel is nontrivial if and only if it contains $v$, i.e., if and only if $(1 + \alpha(v)) \cdot v = 0$.  The result follows immediately.
\end{proof}

\begin{remark}
\label{rmk:transvection}
While the following terminology is not necessary for our proofs, it is worth mentioning the standard classification of reflections in $\GL(V)$. A reflection $s$ with $\beta_s := \det(s) \neq 1$ is called a \emph{semisimple reflection}; all such reflections are diagonalizable.  If $s$ is semisimple then $v_s$ is the non-unit eigenvector of $s$, and $s(v_s) = \beta_s v_s$. 

A reflection $t$ with $\det(t) = 1$ is called a \emph{transvection}; transvections are not diagonalizable.  
If $t$ is a transvection, the vector $v_t$ is called its \emph{transvector} and is in fact an element of the fixed space $V^t$.  By considering the rational canonical form of a transvection, it can be seen that its minimal polynomial is $ (y-1)^2 \in \FF[y]$, and thus $t(t-1)(x) = (t-1)(x)$ for all $x\in V$.  
\end{remark}

\begin{definition}
For an element $g\in G$, we say the ordered tuple $S = (s_1, s_2, \ldots, s_k) \in R^k$ of reflections is an \emph{ordered factorization} of $g$ if $g = s_1s_2\cdots s_k$.  (Note that we do not require $k$ to be minimal here.)  If $S$ is an ordered factorization, set $v_i:= v_{s_i}$ and $\alpha_i:= \alpha_{s_i}$ as in Proposition~\ref{Vectors}, and define the hyperplane $H_i:= V^{s_i} = \ker(\alpha_i)$.  For an ordered factorization $S$, define the \emph{$S$-fixed} and \emph{$S$-moved} spaces 
\[
V^S := \bigcap_{i = 1}^k H_i
\qquad \textrm{and} \qquad
V_S :=  \spn_{\FF}\{v_{1}, \ldots, v_{k}\}.
\]
\end{definition}

We begin the proof of the first main theorem with an easy result relating the spaces $V_S$ and $V^S$ to the fixed and moved spaces of the product of $S$.

\begin{prop}
\label{prop:easy linear preliminary}
If $S = (s_1, \ldots, s_k)$ is an ordered factorization of an element $g$ in $\GL(V)$, then
$\im(g - 1) \subseteq V_S$ and $\ker(g - 1) \supseteq V^S$.
\end{prop}
\begin{proof}
To prove the first statement, we use a telescoping sum: for every $x \in V$ we have
\begin{align*}
g(x) - x & = 
(s_1s_2\cdots s_k(x) - s_2\cdots s_k(x)) + (s_2s_3\cdots s_k(x) - s_3 \cdots s_k(x)) + \ldots + (s_k(x) - x) \\
& = (s_1 - 1)(x_1) + (s_2 - 1)(x_2) + \ldots + (s_k - 1)(x_k)
\end{align*}
for some vectors $x_1, \ldots, x_k$.  The vector $g(x) - x$ represents an arbitrary element of $\im(g - 1)$, and the final expression evidently belongs to $\im(s_1 - 1) + \ldots + \im(s_k - 1) = V_S$, as desired.

For the second statement, observe that if $x \in V^S$ then $s_i(x) = x$ for all $i$, by definition.  Thus in this case $g(x) = s_1 \cdots s_k(x) = x$, as desired.
\end{proof}

We are now prepared to prove the first main result of this paper, giving a simple geometric rule for detecting minimality of reflection factorizations in the general linear group.  The result is illustrated in Example~\ref{example} below.

\begin{theorem}
\label{dual delmas lemma}
\label{thm:elise}
\label{Facts}
Let $V$ be a finite-dimensional vector space, $g$ an arbitrary element of $\GL(V)$, and $S = (s_1, \ldots, s_k)$ a tuple of reflections such that $g = s_1 \cdots s_k$. 
We have that
\begin{enumerate}[\quad(a)]
\item if $\codim V^S = k$ then $\ell_R(g) = \dim V_S$, 
\item if $\dim V_S = k$ then $\ell_R(g) = \codim V^S$, and
\item $\ell_R(g) = k$ if and only if $\codim V^S = \dim V_S = k$.
\end{enumerate}
\end{theorem}

\begin{proof}
For each $i$, fix a pair $(v_i, \alpha_i) \in V \times V^*$ representing $s_i$ as in Proposition~\ref{Vectors}, so in particular $v_i$ belongs to $\im(s_i - 1)$, and let $H_i = V^{s_i} = \ker(\alpha_i)$ be the fixed hyperplane of $s_i$.

We begin with part (a).  Suppose that $k = \codim(V^S) = \codim\left(\bigcap_{i=1}^k H_i\right)$.  Since codimension is subadditive over intersections, we have the chain of strict inclusions 
\[
H_k \supsetneq H_{k - 1}\cap H_k \supsetneq \dots \supsetneq \bigcap_{i=2}^k H_i \supsetneq V^S.
\]
Thus, for each $j$ there is a vector $x_j$ in $\bigcap_{i={j + 1}}^k H_i$ with $x_j \notin H_j$.  For such a vector $x_1$ we have
\begin{align*}
(g-1)(x_1) &= s_1(x_1) - x_1 \\
&= \alpha_1(x_1) \cdot v_{1}
\end{align*}
with $\alpha_1(x_1) \neq 0$.  Thus, $v_1 \in \im(g-1)$.  Similarly, for a vector $x_2 \in \bigcap_{i=3}^k H_i$ with $x_2 \notin H_2$, we have 
\begin{align*}
(g-1)(x_2) &= s_1s_2(x_2)-x_2\\
&= s_1(x_2 + \alpha_{2}(x_2) \cdot v_2) - x_2\\
&= (\alpha_1(x_2)+\alpha_1(v_2)\alpha_2(x_2)) \cdot v_1  + \alpha_2(x_2) \cdot v_2
\end{align*}
with $\alpha_2(x_2)\neq0$.  Since $v_1$ belongs to the subspace $\im(g - 1)$, we have that $\alpha_2(x_2) \cdot v_2\in \im(g-1)$ and so also $v_2\in \im(g-1)$.  Continuing inductively in this way, we have that $v_i \in \im(g-1)$ for all $i$ and thus $V_S \subseteq \im(g-1)$.  Combining this with Proposition~\ref{prop:easy linear preliminary}, it follows that $V_S = \im(g-1)$.  By Theorem~\ref{thm:linear}(c), we have that $\dim(V_S)= \ell_R(g)$, as desired.

Next, we consider part (b); the argument is a dualization of the preceding paragraph.  Suppose that $k = \dim(V_S) = \dim(\spn\{v_1, \ldots, v_k\})$.  Thus, the $v_k$ are linearly independent, and so for each $j$ there is a linear form $\beta_j$ such that $\beta_j(v_1) = \dots = \beta_j(v_{k - j}) = 0$ and $\beta_j(v_{k + 1 - j}) \neq 0$.
Let $S' = (s_1, \ldots, s_{k - 1})$ be the prefix of $S$.  Choose any vector $x$ in $\ker(g - 1)$ and consider $\beta_1(g(x) - x)$.  On one hand, since $x \in \ker(g - 1)$ we have $g(x) - x = 0$ and so $\beta_1(g(x) - x) = 0$.  On the other hand, by a telescoping expansion we have
\[
\beta_1(g(x) - x) = 
\beta_1(s_1s_2\cdots s_k(x) - s_2\cdots s_k(x)) + 
\ldots + 
\beta_1(s_{k - 1}s_k(x) - s_k(x)) + 
\beta_1(s_k(x) - x).
\]
On the right side, the argument of the first application of $\beta_1$ is a multiple of $v_1$, the argument of the second application of $\beta_1$ is a multiple of $v_2$, and so on.  By the choice of $\beta_1$, this implies
\begin{align*}
0 = \beta_1(g(x) - x) & = \beta_1(s_k(x) - x) \\
& = \beta_1( \alpha_k(x) v_k) \\
& = \beta_1(v_k) \alpha_k(x).
\end{align*}
Since $\beta_1(v_k) \neq 0$, it follows that $\alpha_k(x) = 0$ and so that $x \in \ker(s_k - 1)$.  Then applying $\beta_2$ to the same telescoping expansion, using the fact that $s_k(x) - x = 0$, we likewise conclude that $\alpha_{k - 1}(x) = 0$ and so that $x \in \ker(s_{k - 1} - 1)$.  Continuing inductively in this way shows that $x$ belongs to $\ker(s_i - 1)$ for each $i$.  This implies that $\ker(g - 1) \subseteq \bigcap \ker(s_i - 1) = V^S$, and so $\ker(g - 1) = V^S$ by Proposition~\ref{prop:easy linear preliminary}.  Finally, we have by Theorem~\ref{thm:linear}(c) that $\ell_R(g) = \codim V^S$, as claimed.

To prove part (c), suppose first that $\codim\left(\bigcap_{i=1}^k H_i\right) = \dim(V_S)= k$. Then by either part (a) or part (b), $\ell_R(g) = k$.  Conversely, suppose that  $\ell_R(g) = k$.  Clearly $\bigcap_{i=1}^k H_i \subseteq V^g$ and thus $\codim\left(\bigcap_{i=1}^k H_i\right) \geq k$.  However, each $H_i$ is a hyperplane and thus the intersection $\bigcap_{i=1}^k H_i$ can have codimension at most $k$.  Thus $\codim(\bigcap_{i=1}^k H_i) = k$.  From part (a), it follows that $\dim(V_S) = \ell_R(g) = k$, as desired.
\end{proof}

\begin{exmp}
\label{example}
Let $\FF = \FF_5$ and $V = (\FF_5)^3$. We will give three examples of ordered reflection factorizations $S = (s_1, s_2, s_3)$ in $\GL(V)$.  In all three cases, the $S$-fixed space $V^S$ will be trivial (i.e., $\codim(V^S) = \codim(H_1\cap H_2\cap H_3) = 3$), but the dimension of the $S$-moved space $V_S$ will vary. In all examples, for $i = 1, 2, 3$ let $s_i$ be represented by the pair $(v_i, \alpha_i)$ as in Proposition~\ref{Vectors}, with 
\[ 
\alpha_1 = \bpm 1 & 0 &0\epm,
 \hskip 1cm 
\alpha_2 = \bpm 0&1&0\epm,
 \hskip 1cm 
 \alpha_3 = \bpm 0&0&1\epm,
\]
so that $H_1, H_2, H_3$ are the three coordinate hyperplanes.  Thus $V^S = H_1 \cap H_2 \cap H_3 = \{0\}$ has codimension $3$ in all cases.
\begin{enumerate}[(i)]

\item Set $v_1 = v_2 = v_3 = \bpm 1 & 1 & 1\epm^\top$.  The matrices for these reflections are
\[ 
s_1 = \bpm 2&0&0\\1&1&0 \\ 1&0&1\epm, 
\hskip .2cm 
s_2 = \bpm 1&1&0 \\ 0&2&0 \\ 0&1&1\epm, 
\hskip .2cm 
\textrm{and}
\hskip .2cm 
s_3 = \bpm 1&0&1 \\ 0&1&1 \\ 0&0&2 \epm,\]
and their product is
\[
g := s_1 \cdot s_2 \cdot s_3 = \bpm 2 & 2 & 4 \\ 1 & 3 & 4 \\ 1 & 2 & 0\epm.
\]
Then $\dim V_S = 1$.  The element $g$ is a reflection (it fixes the vectors $\bpm 1 & 0 & 1\epm^\top$ and $\bpm 1 & 2 & 0 \epm^\top$) and so $\ell_R(g) = 1$.

\item
Set $v_1 = v_2 = \bpm 1 & 1 & 1\epm^\top$, $v_3 = \bpm 1 & 0 & 1\epm^\top$.  The matrices for these reflections are 
\[
s_1 = \bpm 2&0&0\\1&1&0 \\ 1&0&1\epm, 
 \hskip .2cm 
s_2 = \bpm 1&1&0 \\ 0&2&0 \\ 0&1&1\epm,
 \hskip .2cm
 \textrm{and}
 \hskip .2cm
s_3 = \bpm 1&0&1\\0&1&0\\0&0&2\epm, 
\]
and their product is
\[
g:= s_1 \cdot s_2 \cdot s_3 = \bpm 2 & 2 & 2 \\ 1 & 3 & 1 \\ 1 & 2 & 3\epm.
\]
In this case $\dim V_S = 2$.  The element $g$ fixes the line $\FF \cdot \bpm 1 & 2 & 0 \epm^\top$, so it is not a reflection, but it can be written as a product of two reflections (for example, as $(s_1 \cdot s_2) \cdot s_3$, where one can check that the first factor $s_1 \cdot s_2$ is in fact a reflection) and so $\ell_R(g) = 2$.
 
\item 
Set $v_1 = \bpm 1&0&0\epm^\top$, $v_2 = \bpm 0&1&0\epm^\top$, and $v_3 = \bpm 0&0&1\epm^\top$.  The matrices for these reflections are
\[ 
s_1 = \bpm 2&0&0 \\ 0&1&0\\0&0&1\epm,
 \hskip .2cm 
s_2 = \bpm 1&0&0\\0&2&0\\0&0&1\epm, 
 \hskip .2cm
 \textrm{and}
 \hskip .2cm
s_3 = \bpm 1&0&0\\0&1&0\\0&0&2\epm,
\]
and their product is $2$ times the identity. In this case $\dim V_S = 3$.
\end{enumerate}
\end{exmp}


\section{The affine group}
\label{sec:affine group}

In the first three parts of this section, we restrict to the case that the underlying field $\FF$ is not $\FF_2$.  (Note that this is not a number-theoretic restriction: fields of \emph{characteristic} $2$ are allowed, as long as they have more than $2$ elements.)  The case of $\FF_2$ is discussed in Section~\ref{sec:F_2}.

\subsection{Definitions}

We begin this section by reviewing the geometry of affine spaces and the general affine group; \cite[Ch.~2]{Berger} is one possible reference.

An \emph{affine space} $\VV$ associated to a finite-dimensional vector space $V$ is a set together with a uniquely transitive $V$-action, that is, for every two points $x, y \in \VV$, there is a unique vector $\lambda$ in $V$ such that $x + \lambda = y$.  One may think of the affine space as retaining those parts of the linear structure of a vector space that do not require a fixed origin; in particular, one cannot take sums of elements of $\VV$, but one can take differences (displacements), and the difference between two elements of $\VV$ is a vector in $V$.  To prevent confusion, we refer to elements of $\VV$ as \emph{points} and elements of $V$ as \emph{vectors}.  One may induce a (non-canonical) vector space structure on $\VV$ by choosing a point of $\VV$ to call the origin; with this choice, the vector space structure is isomorphic to $V$.  A subset $X$ of $\VV$ is an \emph{(affine) subspace} if $X = a + U$ for some point $a \in \VV$ and some subspace $U \subseteq V$; in this case the choice of $U$ is unique.

An \emph{affine transformation} $f$ between two affine spaces $\VV_1, \VV_2$ over the same field is a map such that if $x_1, x_2, y_1, y_2$ are points in $\VV$ and $y_1 - x_1 = y_2 - x_2$ then $f(y_1) - f(x_1) = f(y_2) - f(x_2)$, and moreover the map $V_1 \to V_2$ sending $y_1 - x_1 \mapsto f(y_1) - f(x_1)$ is linear.  The \emph{general affine group} (or just \emph{affine group} for short) $\GA(\VV)$ of $\VV$ consists of all invertible affine transformations from $\VV$ to itself.  

Given a point $a$ in $\VV$, the subgroup of all affine transformations that fix $a$ is naturally isomorphic to the general linear group $\GL(V)$: the linear map $g$ corresponds to the affine transformation $\gg$ defined by $\gg(x) := a + g(x - a)$ for every $x \in \VV$; obviously $\gg(a) = a$ in this case.  We denote by $\iota_a$ this \emph{inclusion map} $\GL(V) \hookrightarrow \GA(\VV)$.  There is also a natural \emph{projection map} $\pi : \GA(\VV) \twoheadrightarrow \GL(V)$ sending an affine transformation $\gg$ to the linear map $g$ defined by $g(\lambda) := \gg(a + \lambda) - \gg(a)$ for every $\lambda \in V$ and an arbitrary (alternatively, every) $a \in \VV$.  The kernel of $\pi$ is exactly the set of \emph{translations} of $\VV$, the maps that, for some fixed $\lambda \in V$, send $x \mapsto x + \lambda$ for all $x \in \VV$.  The inclusion and projection maps show that the full affine group $\GA(\VV)$ is isomorphic to the semidirect product $V \rtimes \GL(V)$ of $\GL(V)$ and the group $V$ acting on $\VV$ by translation; every affine transformation can be written uniquely as a linear transformation (with respect to some prescribed origin) followed by a translation.

Let $\FF$ be the field of scalars of $\VV$ and let $n = \dim \VV$. By choosing coordinates for $\VV$, we may realize $\GA(\VV)$ as a set of $(n + 1) \times (n + 1)$ matrices over $\FF$:
\[
\GA(\VV) \cong \left\{ \begin{bmatrix} g & \lambda \\ 0 & 1 \end{bmatrix} \colon g \in \GL_n(\FF), \lambda \in \FF^n \right\} \subset \GL_{n + 1}(\FF).
\]
In this case, $\VV$ is identified with the affine hyperplane $\{(x_1, \ldots, x_n, 1) \}$ and $V$ is identified with the linear hyperplane $\{(x_1, \ldots, x_n, 0) \}$  in $\FF^{n + 1}$.  Given $\gg =  \begin{bmatrix} g & \lambda \\ 0 & 1 \end{bmatrix} \in \GA(\VV)$, the action of $\gg$ on a point $\aff{a} = (a_1, \ldots, a_n, 1) \in \VV$ is $\gg(\aff{a}) = \bpm g(a) + \lambda \\ 1 \epm$ where $a = (a_1, \ldots, a_n)$, i.e., the block $g$ is the matrix of the projection of $\gg$ into $\GL(V)$ and $\lambda$ is the associated translation vector.  

A \emph{reflection} in $\GA(\VV)$ is an element that fixes a subspace of codimension $1$ (a hyperplane) pointwise.  (As before, this is a larger class than the orthogonal reflections, but the reflections in this sense that belong to the isometries of $\RR^n$ are exactly the orthogonal reflections.)  We denote by $\affR$ the subset of reflections of $\GA(\VV)$.  The fact that $\affR$ generates $\GA(\VV)$, and so that it makes sense to speak of the reflection length $\ell_{\affR}$ and to call $\GA(\VV)$ a reflection group, is deferred to Proposition~\ref{prop:reflection facts}(d) below.

The natural inclusion and projection maps $\iota_a$ and $\pi$ send reflections to reflections: if $r$ is a reflection in $\GL(V)$ and $a$ is a point in $\VV$ then $\iota_a(r)$ is a reflection in $\GA(\VV)$, and likewise if $\rr$ is a reflection in $\GA(\VV)$ then $\pi(\rr)$ is a reflection in $\GL(V)$.  In what follows, we will continue to use $R$ to denote the set of reflections in $\GL(V)$.

\subsection{Fundamental subspaces, tripartite classification}

We define three fundamental subspaces of an element of $\GA(\VV)$.  
\begin{definition*}
For an element $\gg \in \GA(\VV)$, its \emph{moved space} is the set 
\[
\mov(\gg) := \{\gg(x) - x \colon x \in \VV\},
\]
its \emph{affine fixed space} is the set 
\[
\fixaff(\gg) := \{x \in \VV \colon \gg(x) = x \},
\]
and its \emph{linear fixed space} is the set 
\[
\fixlin(\gg) := \{ v \in V \colon g(v) = v\} = \ker\left( g - 1 \right),
\]
where $g = \pi(\gg)$ is the projection of $\gg$ into $\GL(V)$.
\end{definition*}
The affine fixed space $\fixaff(\gg)$ is a subspace of the affine space $\VV$, the linear fixed space $\fixlin(\gg)$ is a (linear) subspace of the vector space $V$, and the moved space $\mov(\gg)$ is an affine subspace of the vector space $V$.

\begin{example}
Let $\VV = \FF_3^2$ be a (coordinatized) two-dimensional affine space over the field $\FF_3$ having three elements, with coordinates $(x, y)$.  We consider three elements $r, s, t$ of $\GA(\VV)$, defined as follows:
\begin{compactitem}
\item $r$ is the reflection defined by $r (x, y) = (x, -y)$;
\item $t$ is the translation defined by $t (x, y) = (x + 1, y)$; and
\item $s$ is the map defined by $s  (x, y) = (y + 1, x)$.
\end{compactitem}
As matrices in $\GL_3(\FF_3)$, we have
\[
r = \begin{bmatrix} 
1 & 0 & 0 \\
0 &-1 & 0 \\
0 & 0 & 1
\end{bmatrix},
\qquad
t = \begin{bmatrix} 
1 & 0 & 1 \\
0 & 1 & 0 \\
0 & 0 & 1
\end{bmatrix},
\qquad \textrm{ and } \qquad
s = \begin{bmatrix} 
0 & 1 & 1 \\
1 & 0 & 0 \\
0 & 0 & 1
\end{bmatrix}.
\]
Then 
\begin{compactitem}
\item $\fixaff(r)$ is the line $\{(x, 0)\}$ in $\VV$, $\fixlin(r)$ is the line $\{(x, 0)\}$ in $V$, and $\mov(r)$ is the (linear) line $\{(0, y)\}$ in $V$; 
\item $\fixaff(t)$ is the empty set, $\fixlin(t) = V$, and $\mov(t)$ is the point $\{(1, 0)\}$ in $V$; and
\item $\fixaff(s)$ is the empty set, $\fixlin(s)$ is the line $\{(x, x)\}$ in $V$, and $\mov(s)$ is the affine line $(1, 0) + \{(x, -x)\}$ in $V$.
\end{compactitem}
The three elements have reflection lengths $\ell_{\affR}(r) = 1 = \dim \mov(r)$, $\ell_{\affR}(s) = 2 = 1 + \dim\mov(s)$, and $\ell_{\affR}(t) = 2 = 2 + \dim\mov(t)$.
\end{example}

This example inspires the following three-part classification of elements of $\GA(\VV)$.  (The names are a second-hand borrowing from the theory of metric spaces of non-positive curvature, via \cite{BradyMcCammond} -- see \cite[p.~230]{BridsonHaefliger} for an explanation of the terminology in its original context.)
\begin{definition}
\label{def:tripartite}
When $|\FF| > 2$, we say that an element $\gg$ of $\GA(\VV)$ is
\begin{compactitem}
\item \emph{elliptic}, if $\gg$ fixes any point of $\VV$ (equivalently, if $\fixaff(\gg)$ is nonempty);
\item \emph{hyperbolic}, if $\gg$ is a nontrivial translation (equivalently, if $\mov(\gg) = \{\lambda\}$ for $\lambda \neq 0$, or if $\fixaff(\gg) = \varnothing$ and $\fixlin(\gg) = V$); and
\item \emph{parabolic}, otherwise.
\end{compactitem}
\end{definition}

\subsection{The main theorem}

Our second main theorem is that the classification of Definition~\ref{def:tripartite} explains the relationship between the reflection length and dimension of the moved space for elements of $\GA(\VV)$.

\begin{theorem}
\label{main affine theorem}
For an element $\gg \in \GA(\VV)$, we have
\[
\ell_{\affR}(\gg) = 
\begin{cases}
\dim\mov(\gg) & \textrm{ if } \gg \textrm{ is elliptic,} \\
\dim\mov(\gg) + 1 & \textrm{ if } \gg \textrm{ is parabolic, and } \\
\dim\mov(\gg) + 2 & \textrm{ if } \gg \textrm{ is hyperbolic.} \\
\end{cases}
\]
\end{theorem}

The rest of this section is devoted to the proof of this result.
We begin with some basic technical facts about how the fixed and moved spaces of elements of $\GA(\VV)$ relate to those of their images in $\GL(V)$.

\begin{prop}
\label{prop:elliptic facts}
Suppose that $\gg$ is an element of $\GA(\VV)$ and that $g$ is its image under the projection map $\pi: \GA(\VV) \twoheadrightarrow \GL(V)$.  The fundamental subspaces of $\gg$ and $g$ are related in the following way:
\begin{compactenum}[(a)]
\item $\mov(\gg) = \im(g - 1) + (\gg(a) - a)$ for any point $a \in \VV$, and
\item if $\gg$ is elliptic then $\fixaff(\gg) = a + \ker(g - 1) = a + \fixlin(\gg)$ for any point $a$ in $\fixaff(\gg)$.
\end{compactenum}
\end{prop}
\begin{proof}
For any point $a$ in $\VV$ and vector $v$ in $V$ we have $\gg(a + v) - (a + v) = \gg(a) - a + g(v) - v$.  Fix $a$ and let $v$ vary; then $a + v$ varies over all of $\VV$, so the set of left sides of the previous equation is exactly $\mov(\gg)$.  On the other hand, as $v$ varies over $V$, the set of right sides is exactly $\gg(a) - a + \im(g - 1)$.  This proves part (a).  For part (b), suppose that $\gg$ is elliptic.  Fix a point $a$ in $\fixaff(\gg)$.  For any point $b$ in $\VV$, we have $\gg(b) = \gg(a) + g(b - a) = a + g(b - a)$, so $b$ belongs to $\fixaff(\gg)$ if and only if $b - a \in \ker(g - 1) = \fixlin(\gg)$, as claimed.
\end{proof}

\begin{remark}
\label{rmk:elliptic moved space}
Part (a) is particularly nice when $\gg$ is elliptic: choosing $a$ to be a fixed point of $\gg$, it says that in this case $\mov(\gg) = \im(g - 1)$.
\end{remark}

Next, we collect some basic facts about reflections in $\GA(\VV)$.
\begin{prop}
\label{prop:reflection facts}
For every reflection $\rr$ in $\GA(\VV)$, 
\begin{enumerate}[\quad(a)]
\item there exists a point $a \in \VV$, a vector $v \in V$, and a linear form $\alpha \in V^*$ such that 
\[
\rr(x) = x + \alpha(x - a) \cdot v
\]
for all $x$ in $\VV$, and
\item $\mov(\rr)$ is a (linear) line in $V$.
\end{enumerate}
In addition,
\begin{enumerate}[\quad(a)]
\addtocounter{enumi}{2}
\item if $\aff{H}$ is a hyperplane in $\VV$ and $a$, $b$ are points in $\VV \smallsetminus \aff{H}$ then there is a unique reflection $\aff{r}$ in $\GA(\VV)$ such that $\fixaff(\aff{r}) = \aff{H}$ and $\aff{r}(a) = b$, and
\item the group $\GA(\VV)$ is generated by its subset of reflections.
\end{enumerate}
\end{prop}
\begin{proof}
Let $\rr$ be a reflection in $\GA(\VV)$, and let $r = \pi(\rr)$ be the projection of $\rr$ into $\GL(V)$.  By Proposition~\ref{Vectors}, there is a linear form $\alpha$ and a vector $v$ such that $r(\lambda) = \lambda + \alpha(\lambda)\cdot v$ for every vector $\lambda \in V$.  Choose a point $a \in \fixaff(\rr)$.  For any point $x$ in $\VV$, we have by definition of $\pi$ that $\rr(x) = \rr(a) + r(x - a) = a + r(x - a)$.  Thus $\rr(x) = a + (x - a) + \alpha(x - a) \cdot v = x + \alpha(x - a) \cdot v$, which is part (a).  Part (b) follows immediately either from part (a) or from Remark~\ref{rmk:elliptic moved space}.

For part (c), let $\aff{H}$, $a$, $b$ be given, and fix a point $c \in \aff{H}$.  Let $H = \{x - y : x, y \in \aff{H}\}$ be the hyperplane in $V$ parallel to $\aff{H}$.  Since the point $a$ does not belong to $\aff{H}$, we have that the vector $a - c$ does not belong to $H$.  Then let $\alpha$ be the unique linear form on $V$ such that $\alpha|_H = 0$ and $\alpha(c - a) = 1$.  By construction, $\alpha(a - b) = \alpha(c - b) - \alpha(c - a) = \alpha(c - b) - 1 \neq -1$, and so by the final claim of Proposition~\ref{Vectors} we have that the map $\lambda \mapsto \lambda + \alpha(\lambda) \cdot (a - b)$ is a reflection in $\GL(V)$.  Using the inclusion map $\iota_c$, this map lifts to the reflection $\rr(x) := x + \alpha(x - c) \cdot (a - b)$ in $\GA(\VV)$.  By construction, $\rr$ fixes $\aff{H}$ and $\rr(a) = b$, as desired.  For uniqueness, it is enough to observe that the resulting reflection $\rr$ does not depend on the choice of a point $c$ nor on the choice of a particular nonzero value for $\alpha(c - a)$ (in the latter case, because the vector $v:= a - b$ will rescale to compensate).

Finally, for part (d), consider an arbitrary element $\gg$ of $\GA(\VV)$.  Choose a point $a \in \VV$ and a hyperplane $\aff{H}$ in $\VV$ that does not include either $a$ or $\gg(a)$.  If $\gg(a) \neq a$, define $\rr$ to be the reflection (guaranteed by part (c)) sending $a \mapsto \gg(a)$ and fixing $\aff{H}$; otherwise, if $\gg(a) = a$, let $\rr$ be the identity.  In either case, $\rr$ is a product of ($1$ or $0$) reflections, and $\gg' := \rr^{-1} \cdot \gg$ is an elliptic element fixing $a$.  Let $g' = \pi(\gg')$ be the linear part of $\gg'$.  Since $g'$ belongs to $\GL(V)$, it can be written as a product of reflections in $R$.  Apply the inclusion map $\iota_a$ to this factorization; it produces an $\affR$-factorization of $\iota_a(g') = \gg'$.  It follows immediately that $\gg = \rr \cdot \gg'$ can be written as a product of reflections, as desired.
\end{proof}

The next result is a final technical lemma, whose second half is a first step towards the proof of the main result.

\begin{prop}
\label{boring technical lemma}
\begin{enumerate}[\quad(a)]
\item For any $\gg_1, \gg_2$ in $\GA(\VV)$, we have $\mov(\gg_1 \cdot \gg_2) \subseteq \mov(\gg_1) + \mov(\gg_2)$.
\item For any $\gg$ in $\GA(\VV)$, we have $\ell_{\affR}(\gg) \geq \dim \mov(\gg)$.
\end{enumerate}
\end{prop}
\begin{proof}
Every element of $\mov(\gg_1 \cdot \gg_2)$ is of the form $\gg_1(\gg_2(a)) - a$ for some point $a$ in $\VV$.  For every point $a$ in $\VV$, we have
\[
\gg_1(\gg_2(a)) - a = 
\left(\gg_1(\gg_2(a)) - \gg_2(a)\right)
+ \left(\gg_2(a) - a\right)
\in
\mov(\gg_1) + \mov(\gg_2),
\]
and part (a) follows immediately.

For part (b), choose a minimal reflection factorization $\gg = \aff{r_1} \cdots \aff{r_k}$ of $\gg$.  By applying part (a) repeatedly,
$\mov(\gg) \subseteq \mov(\aff{r}_1) + \ldots + \mov(\aff{r}_k)$.  By Proposition~\ref{prop:reflection facts}(b), the right side is of dimension at most $k = \ell_{\affR}(\gg)$, so taking dimensions gives the result.
\end{proof}

The next five results collectively establish the second main theorem for fields of size larger than $2$.

\begin{prop}
\label{prop:if elliptic then right length}
If $\gg$ in $\GA(\VV)$ is elliptic then $\ell_{\affR}(\gg) = \dim\mov(\gg)$.
\end{prop}
\begin{proof}
Suppose that $\gg$ is elliptic; say that point $a \in \VV$ satisfies $\gg(a) = a$.  Consider the subgroup $G_a \cong \GL(V)$ of $\GA(\VV)$ that fixes $a$, and let $R_a = \affR \cap G_a$ be the subset of reflections in $G_a$.  Since every $R_a$-factorization of $\gg$ is also an $\affR$-factorization, we have $\ell_{\affR}(\gg) \leq \ell_{R_a}(\gg)$.  On the other hand, choose a shortest $\affR$-factorization 
\[
\gg = \aff{r_1} \cdots \aff{r_k}
\]
of $\gg$.  Projecting both sides into $\GL(V)$ gives an associated $R$-factorization
\[
g = r_1 \cdots r_k.
\]
But under the inclusion $\iota_a: \GL(V) \overset{\sim}{\longrightarrow} G_a \subset \GA(\VV)$, $g$ is sent to $\gg$ and each of the $r_i$ is sent to a reflection in $R_a$.  Thus, the factorization of $g$ as a product of reflections in $\GL(V)$ lifts to an $R_a$-factorization of $\gg$ of length $k$.  It follows that $\ell_{R_a}(\gg) \leq k = \ell_{\affR}(\gg)$.  Putting these two inequalities together and using the isomorphism $G_a \cong \GL(V)$, Theorem~\ref{thm:linear}(c), and Proposition~\ref{prop:elliptic facts}(a), we have
\[
\ell_{\affR}(\gg) = \ell_{R_a}(\gg) = \ell_{R}(g) = \dim \im(g - 1) = \dim \mov(\gg),
\]
as claimed.
\end{proof}

\begin{prop}
\label{prop:elliptic converse}
If $\ell_{\affR}(\gg) = \dim\mov(\gg)$ then $\gg$ is elliptic.
\end{prop}
\begin{proof}

Suppose that $\gg = \aff{r_1} \cdots \aff{r_k}$ is a minimal $\affR$-factorization of $\gg$ and that $\dim\mov(\gg) = k$.  Projecting both sides into $\GL(V)$ gives a linear reflection factorization
\begin{equation}
\label{eq:projected factorization}
g = r_1 \cdots r_k.
\end{equation}
Consider the moved space $\im(g - 1)$ of the linear map $g$ acting on $V$.  By Proposition~\ref{prop:elliptic facts}(a), $\dim \im(g - 1) = \dim \mov(\gg) = k$.  Therefore, by Theorem~\ref{thm:linear}(c), the factorization \eqref{eq:projected factorization} is a shortest reflection factorization.  It follows from Theorem~\ref{thm:elise} that the linear forms defining the fixed hyperplanes of the $r_i$ are linearly independent in the dual space $V^*$.  Lifting back to the affine setting, the directions of the fixed planes of the $\aff{r_i}$ are linearly independent.  But any collection of independent hyperplanes in affine space has non-empty intersection, i.e., $\bigcap_k \fixaff(\aff{r_i})$ contains some point $a$.  But then $\gg(a) = \aff{r_1} \cdots \aff{r_k}(a) = a$, so $a \in \fixaff(\gg)$ and so $\gg$ is elliptic, as claimed.
\end{proof}

\begin{prop}
\label{prop:reflection length is not too long}
For every $\gg$ in $\GA(\VV)$, we have $\ell_{\affR}(\gg) \leq 2 + \dim\mov(\gg)$.
\end{prop}
\begin{proof}
Let $\gg$ be an arbitrary element of $\GA(\VV)$.  As in the proof of Proposition~\ref{prop:reflection facts}(d), there is a reflection $\aff{r} \in \affR$ such that the map $\aff{r} \cdot \gg$ is elliptic and so 
\[
\ell_{\affR}(\gg) = \ell_{\affR}\left(\aff{r}^{-1} \cdot (\aff{r} \cdot \gg)\right)
  \leq 1 + \ell_{\affR}(\aff{r} \cdot \gg)
  = 1 + \dim \mov(\aff{r} \cdot \gg)
\]
by the triangle inequality and Proposition~\ref{prop:if elliptic then right length}.  Finally, since $\dim \mov(\aff{r}) = 1$, we have by Proposition~\ref{boring technical lemma}(a) that $\dim \mov(\aff{r} \cdot \gg) \leq 1 + \dim \mov(\gg)$, and the claim follows.
\end{proof}

For each non-elliptic element $\gg$ in $\GA(\VV)$, it follows from Propositions~\ref{boring technical lemma}(b),~\ref{prop:elliptic converse}, and~\ref{prop:reflection length is not too long} that $\ell_{\affR}(\gg)$ is equal to either $\dim\mov(\gg) + 1$ or $\dim\mov(\gg) + 2$. The next two results distinguish these cases.

\begin{prop}
\label{prop:if parabolic then right length}
If $\gg$ is parabolic then $\ell_{\affR}(\gg) = 1 + \dim\mov(\gg)$.
\end{prop}
\begin{proof}
We refine the proof of Proposition~\ref{prop:reflection length is not too long}.  Suppose that $\gg$ in $\GA(\VV)$ is parabolic. Since $\gg$ is not a translation, its linear fixed space $\fixlin(\gg)$ is a proper subspace of $V$.  Let $H$ be a hyperplane in $V$ containing $\fixlin(\gg)$. Choose any point $a \in \VV$.  Since $|\FF| > 2$, it follows that the two hyperplanes $a + H$ and $\gg(a) + H$ do not cover $\VV$.  Let $b$ be a point not contained in either of these planes, so that $\aff{H} := b + H$ is a hyperplane in $\VV$ that does not contain either $a$ or $\gg(a)$. Thus, by Proposition~\ref{prop:reflection facts}(c), there is a reflection $\aff{r}$ sending $\gg(a) \mapsto a$ and fixing $\aff{H}$.  

By construction, $(\aff{r} \cdot \gg)(a) = a$, so $\aff{r} \cdot \gg$ is elliptic.  Moreover, $\fixlin(\aff{r}) = H \supseteq \fixlin(\gg)$, so $\fixlin(\aff{r} \cdot \gg) \supseteq \fixlin(\gg)$.  It follows that $\dim \fixlin(\aff{r} \cdot \gg) \geq \dim \fixlin(\gg)$.  By Proposition~\ref{prop:elliptic facts}(a) and the rank-nullity theorem, the moved space and linear fixed space of every element of $\GA(\VV)$ have complementary dimensions; it follows that $\dim \mov(\aff{r} \cdot \gg) \leq \dim \mov(\gg)$.  Therefore
\[
\ell_{\affR}(\gg) \leq 1 + \dim \mov(\aff{r} \cdot \gg) \leq 1 + \dim \mov(\gg),
\]
and the result follows by Propositions~\ref{boring technical lemma}(b) and~\ref{prop:elliptic converse}.
\end{proof}

\begin{prop}
\label{prop:parabolic converse}
If $\gg$ is hyperbolic then $\ell_{\affR}(\gg) = 2 + \dim\mov(\gg)$.
\end{prop}
\begin{proof}
If $\gg$ is hyperbolic then by definition $\gg$ is a translation, $\mov(\gg)$ is a singleton set, and $\dim\mov(\gg) = 0$.  Thus, by Propositions~\ref{boring technical lemma}(b) and~\ref{prop:reflection length is not too long}, we have $\ell_{\affR}(\gg) \in \{0, 1, 2\}$.  Since $\gg$ is not the identity, $\ell_{\affR}(\gg) \neq 0$, and since $\gg$ is not a reflection, $\ell_{\affR}(\gg) \neq 1$.
\end{proof}

Finally, Theorem~\ref{main affine theorem} follows immediately from Propositions~\ref{prop:if elliptic then right length},~\ref{prop:if parabolic then right length}, and~\ref{prop:parabolic converse}.

\subsection{The case of the field of two elements}
\label{sec:F_2}

In this section, we detail the ways in which the preceding story changes over the field with two elements.  As a first step, we give a new version of the tripartite classification (Definition~\ref{def:tripartite}) of elements of $\GA(\VV)$.

\begin{definition}
\label{def:tripartite for F_2}

For any field $\FF$, we say that an element $\gg$ of $\GA(\VV)$ is
\begin{compactitem}
\item \emph{elliptic}, if $\gg$ fixes any point of $\VV$;
\item \emph{hyperbolic}, if $\fixaff(\gg) = \varnothing$ and $(a + \fixlin(\gg)) \cup (\gg(a) + \fixlin(\gg)) = \VV$ for every point $a$ in $\VV$; and
\item \emph{parabolic}, otherwise.
\end{compactitem}
\end{definition}

Observe that in the case $|\FF| > 2$, this definition is equivalent to Definition~\ref{def:tripartite}: every translation has linear fixed space $V$ and so is hyperbolic under the new definition, while for any $\gg \in \GA(\VV)$ such that $\fixaff(\gg) = \varnothing$ and $\fixlin(\gg) \neq V$, we have that $(a + \fixlin(\gg)) \cup (\gg(a) + \fixlin(\gg))$ is a union of two proper subspaces of $\VV$ and so is not equal to $\VV$.  However, over $\FF_2$, it is possible for a union of two proper affine subspaces to equal the full affine space $\VV$.

\begin{example}
\label{ex:hyperbolic non-translation}
Let $\VV$ be the (coordinatized) affine space $\FF_2^2$, and let $\gg$ be the map sending $(x, y) \mapsto (x + 1, x + y)$.  Thus $\gg$ may be represented by the matrix $\begin{bmatrix} 1 & 0 & 1 \\ 1 & 1 & 0 \\ 0 & 0 & 1\end{bmatrix}$ in $\GL_3(\FF_2)$.  Then $\fixaff(\gg) = \varnothing$ and $\fixlin(\gg)$ is the line $x  = 0$ in $V$.  For any point $a = (a_1, a_2) \in \VV$, it follows that $(a + \fixlin(\gg)) \cup (\gg(a) + \fixlin(\gg)) = \{(x, y) : x = a_1\} \cup \{(x, y) : x = a_1 + 1\} = \VV$.  Thus this element $\gg$ is a non-translative hyperbolic element.
\end{example}

We now briefly catalogue how this changed definition affects the results of the preceding subsections.  Proposition~\ref{prop:elliptic facts}, Remark~\ref{rmk:elliptic moved space}, and Proposition~\ref{prop:reflection facts}(a, b, c) are entirely valid over $\FF_2$.  
Proposition~\ref{prop:reflection facts}(d), that $\GA(\VV)$ is generated by reflections, is true \emph{except} in the case that $\VV$ is a $1$-dimensional affine space over $\FF_2$.  In this case, the group $\GA(\VV)$ has two elements, the identity and a translation; in particular, it contains no reflections at all.  The given proof breaks down in the following step: for the non-identity element $\gg$ and a point $a$ in $\VV$, there is no hyperplane (point) $\aff{H}$ that does not intersect $a$ or $\gg(a)$.  (When the underlying field has size larger than $2$, finding a suitable $\aff{H}$ is uninteresting: every parallelism class of hyperplanes contains at least $3$ planes, while $a$ and $\gg(a)$ each belong to exactly one member of the class.)  When $\FF = \FF_2$ and the dimension of $\VV$ is larger than $1$, it again is possible to make the choice of a hyperplane not containing $\gg(a)$ or $a$: let $H$ be a hyperplane in $V$ that contains the vector $\gg(a) - a$, so that one of the two affine hyperplanes in $\VV$ that are parallel to it contains both $a$ and $\gg(a)$, leaving the other to contain neither.  Thus, going forward we restrict our analysis to the case that $\VV$ has dimension at least $2$.

The proofs of Propositions~\ref{boring technical lemma},~\ref{prop:if elliptic then right length}, and~\ref{prop:elliptic converse} (establishing that the correctness of Theorem~\ref{main affine theorem} for elliptic elements in $\GA(\VV)$, as well as the general lower bound $\ell_{\affR}(\gg) \geq \dim \mov(\gg)$ for every element $\gg$ in $\GA(\VV)$) are valid as written when $\FF = \FF_2$.  The proof of Proposition~\ref{prop:reflection length is not too long} relies on the choice of a hyperplane not passing through two given points, which (as in the previous paragraph) is fine when $\FF = \FF_2$ and $\dim(\VV) > 1$; the rest of the proof is valid as written over $\FF_2$.

The proof of Proposition~\ref{prop:if parabolic then right length} explicitly invokes the condition $|\FF| > 2$ in order to conclude that two proper subspaces do not cover $\VV$.  Thus, we give here a complete version of the proof over $\FF_2$ (using Definition~\ref{def:tripartite for F_2} in place of Definition~\ref{def:tripartite}).
\begin{proof}[Proof of Proposition~\ref{prop:if parabolic then right length} over the field $\FF_2$.]
Suppose that $\gg$ in $\GA(\VV)$ is parabolic. By Definition~\ref{def:tripartite for F_2}, there exists a point $a$ in $\VV$ such that $U := (a + \fixlin(\gg)) \cup (\gg(a) + \fixlin(\gg)) \subsetneq \VV$.  Since we are working over $\FF_2$, we have that 
\[
U = a + \{0, \gg(a) - a\} + \fixlin(\gg) = a + \left( \FF_2 \cdot (\gg(a) - a) + \fixlin(\gg)\right)
\]
is an affine subspace of $\VV$, and so in fact a proper affine subspace.   Thus there is some hyperplane $\aff{H}$ containing $U$, some point $b$ in $\VV \smallsetminus \aff{H}$, and a hyperplane $\aff{H}' := \aff{H} + (b - a)$ not intersecting $\aff{H}$ and so not containing $a$ or $\gg(a)$.  


The remainder of the proof is the same: by Proposition~\ref{prop:reflection facts}(c), there is a reflection $\aff{r}$ sending $\gg(a) \mapsto a$ and fixing $\aff{H}$.  
By construction, $(\aff{r} \cdot \gg)(a) = a$, so $\aff{r} \cdot \gg$ is elliptic.  Moreover, $\fixlin(\aff{r}) = H \supseteq \fixlin(\gg)$, so $\fixlin(\aff{r} \cdot \gg) \supseteq \fixlin(\gg)$.  It follows that $\dim \fixlin(\aff{r} \cdot \gg) \geq \dim \fixlin(\gg)$.  By Proposition~\ref{prop:elliptic facts}(a) and the rank-nullity theorem, the moved space and linear fixed space of every element of $\GA(\VV)$ have complementary dimensions; it follows that $\dim \mov(\aff{r} \cdot \gg) \leq \dim \mov(\gg)$.  Therefore
\[
\ell_{\affR}(\gg) \leq 1 + \dim \mov(\aff{r} \cdot \gg) \leq 1 + \dim \mov(\gg),
\]
and the result follows by Propositions~\ref{boring technical lemma}(b) and~\ref{prop:elliptic converse}.
\end{proof}

Over any field, the proof of Proposition~\ref{prop:parabolic converse} is a valid proof of the fact that translations have reflection length $2$.  However, as Example~\ref{ex:hyperbolic non-translation} shows, there are additional hyperbolic elements over $\FF_2$.  
 Thus, we give here an extension of the proof over $\FF_2$ (using Definition~\ref{def:tripartite for F_2} in place of Definition~\ref{def:tripartite}).
\begin{proof}[Proof of Proposition~\ref{prop:parabolic converse} over the field $\FF_2$.]
In light of Propositions~\ref{boring technical lemma}(b),~\ref{prop:reflection length is not too long}, and~\ref{prop:elliptic converse}, it is equivalent to show that if $\gg$ is an element of $\GA(\VV)$ such that $\ell_{\affR}(\gg) = \dim \mov(\gg) + 1$ then $\gg$ is parabolic.  Fix an element $\gg \in \GA(\VV)$ such that $\dim\mov(\gg) = k$ and $\ell_{\affR}(\gg) = k + 1$ for some nonnegative integer $k$.  By Proposition~\ref{prop:if elliptic then right length}, $\gg$ is not elliptic, so $\fixaff(\gg) = \varnothing$.  Let $g:= \pi(\gg)$ be the projection of $\gg$ into $\GL(V)$.  Since $\gg$ is not elliptic, we have by Proposition~\ref{prop:elliptic facts}(a) that $\mov(\gg)$ is a nontrivial translation of $\im(g - 1)$.  In particular, the two have the same dimension: $k = \dim\mov(\gg) = \dim\im(g - 1) = \ell_{R}(g)$ (where the last inequality follows from Theorem~\ref{thm:linear}(c)).
Write
\[
\gg = \aff{r}_1 \cdots \aff{r}_k \cdot \aff{r}_{k + 1}
\]
for some affine reflections $\aff{r}_i$, and let $r_i := \pi(\aff{r}_i)$ be the projection of $\aff{r}_i$ for each $i$.  Then 
\[
g = r_1 \cdots r_k \cdot r_{k + 1}
\]
is a slightly-longer-than-minimal reflection factorization of $g$.  

By Proposition~\ref{Vectors}, we may choose for $i = 1, \ldots, k + 1$ a nonzero vector $v_i$ in the moved space of $\aff{r}_i$ (equivalently, $r_i$) and a nonzero linear functional $\alpha_i$ such that
\[
r_i(\lambda) = \lambda + \alpha_i(\lambda) \cdot v_i
\]
for all $\lambda$ in $V$. By Proposition~\ref{boring technical lemma}(a), we have
\[
\mov(\gg) \subseteq \mov(r_1) + \ldots + \mov(r_{k + 1}) = \spn\{v_1, \ldots, v_{k + 1}\}.
\]
The left side is a $k$-dimensional affine-but-not-linear subspace of $V$, while the right side is a linear subspace spanned by a set of $k + 1$ vectors; it follows that actually these $k + 1$ vectors must be linearly independent.  Therefore, by Theorem~\ref{dual delmas lemma}(b), the fixed spaces of $r_1, \ldots, r_{k + 1}$ must have intersection with codimension exactly $k$, and so the linear forms $\alpha_1, \ldots, \alpha_{k + 1}$ span a subspace of the dual space $V^*$ having dimension exactly $k$.  Choose a $k$-element subset of $[k + 1]$ so that the associated $\alpha_i$ form a basis for their span; say that $\alpha_m$ is the omitted element.  The element $(\rr_{m + 1} \cdots \rr_{k + 1}) \cdot \gg \cdot (\rr_{m + 1} \cdots \rr_{k + 1})^{-1}$ is conjugate to $\gg$ and so has the same reflection length and dimension of moved space; moreover, it can be factored into reflections as $\rr_{m + 1} \cdots \rr_{k + 1} \cdot \rr_1 \cdots \rr_{m}$, where the linear forms associated to the first $k$ reflections are linearly independent.  Thus, replacing $\gg$ by this element, we may assume without loss of generality that $m = k + 1$, that $\alpha_1, \ldots, \alpha_k$ are linearly independent, and that $\alpha_{k + 1}$ is in their span.  Then define $\gg' = \aff{r}_1 \cdots \aff{r}_k$.  Since the fixed planes of the factors on the right side are in linearly independent directions, they have nonempty intersection.  Thus $\gg'$ is an elliptic element of $\GA(\VV)$, having $\dim \mov(\gg') = \ell_{\affR}(\gg') = k$ and $\gg = \gg' \cdot \aff{r}_{k + 1}$.

Since $\gg'$ is elliptic and $\aff{r}_1 \cdots \aff{r}_k$ is a shortest reflection factorization, we have $\fixlin(\gg') =  \bigcap_{i = 1}^k \fixlin(\aff{r}_i) = \bigcap_{i = 1}^k \ker(\alpha_i)$ is a subspace of codimension $k$.  Since $\alpha_{k + 1}$ is in the span of $\alpha_1, \ldots, \alpha_k$, we have that $H := \ker(\alpha_{k + 1}) = \fixlin(\aff{r}_{k + 1})$ contains $\fixlin(\gg')$.  It follows that $\fixlin(\gg)$ contains $\fixlin(\gg')$.  Moreover, by definition $\fixlin(\gg) = \ker(g - 1)$, and so $\codim \fixlin(\gg) = \dim \im(g - 1) = k$, so in fact $\fixlin(\gg) = \bigcap_{i = 1}^k \ker(\alpha_i) \subseteq H$.

Pick a point $b \in \fixaff(\gg')$, and let $a := \aff{r}_{k + 1}^{-1}(b)$.  We claim that $a$ witnesses the fact that $\gg$ is not hyperbolic.  Observe first that
\[
\gg(a) = \gg' ( \aff{r}_{k + 1}(a)) = \gg'(b) = b = \aff{r}_{k + 1}(a).
\]
Therefore
\begin{align*}
(a + \fixlin(\gg)) \cup (\gg(a) + \fixlin(\gg))
& \subset
(a + H) \cup (\gg(a) + H) \\
& =
(a + H) \cup (\aff{r}_{k + 1}(a) + H)
.
\end{align*}
Since $\gg$ is not elliptic, $a \neq \aff{r}_{k + 1}(a)$, and since $\aff{r}_{k + 1}$ is invertible it follows that neither $a$ nor $\aff{r}_{k + 1}(a)$ belongs to $\fixaff(\aff{r}_{k + 1})$.  Therefore the right side of the last equation does not include this fixed space and so does not cover $\VV$.  Since $\gg$ is neither elliptic nor hyperbolic, it is parabolic, as claimed.
\end{proof}

Finally, the preceding results show that, under Definition~\ref{def:tripartite for F_2}, the result of Theorem~\ref{main affine theorem} is valid over every field.

\section{Further remarks}
\label{sec:remarks}

\subsection{Precise statement of other affine results}
\label{sec:affine statements}

As promised in the introduction, we give here the statements of the theorems of Brady--McCammond and Lewis--McCammond--Petersen--Schwer on the reflection length of elements in the group of isometries of real Euclidean space and in an affine Coxeter group, respectively.  (For definitions and terminology related to Coxeter groups, we refer the reader to Humphreys's text \cite{Humphreys}, in particular to Chapter 4.)  These results should be compared with the statement of Theorem~\ref{main affine theorem}.

\begin{thm}[{\cite[Thm.~5.7]{BradyMcCammond}}]
\label{thm:BM}
Let $G$ denote the group of isometries of real Euclidean space $\RR^n$, let $R$ denote the set of reflections in $G$, and let $g$ be an arbitrary element of $G$.  Then the reflection length of $g$ is
\[
\ell_R(g) = \begin{cases}
\dim(\mov(g)) & \textrm{if } g \textrm{ fixes a point, and} \\
\dim(\mov(g)) + 2 & \textrm{otherwise}.
\end{cases}
\]
\end{thm}

\begin{thm}[{\cite[Thm.~A]{LMPS}}]
\label{thm:LMPS}
Let $G$ be an affine Coxeter group acting on real Euclidean space $\RR^n$, with reflections $R$, associated finite Coxeter group $G_0$, and projection map $\pi: G \onto G_0$, and let $g$ be an arbitrary element of $G$.  Then the reflection length of $g$ is 
\[
\ell_R(g) = \dim(\mov(g)) + 2d(g)
\]
where $d(g) := \dim_{\mathrm{lin}}(\mov(g)) - \dim(\mov(g))$ and 
for any subset $X$ of $\RR^n$, $\dim_{\mathrm{lin}}(X)$ is defined to be the smallest dimension of a moved space of an element in $G_0$ that contains $X$.
\end{thm}

\begin{remark}
Oddly, the formula in Theorem~\ref{thm:LMPS} also gives the correct result in Theorem~\ref{thm:BM} if one takes $G_0$ to be the orthogonal group $O_n(\RR)$: 
for an element $g$ with a fixed point, $\mov(g)$ is a linear subspace of $\RR^n$ and so $\dim_{\mathrm{lin}}(\mov(g)) = \dim(\mov(g))$ and $d(g) = 0$, while if $g$ does not have a fixed point then $\mov(g)$ is a (non-linear) affine subspace, the dimension of the smallest linear subspace of $\RR^n$ containing $\mov(g)$ is $\dim(\mov(g)) + 1$, and every linear subspace of $\RR^n$ is a moved space of some element in $G_0 = O_n(\RR)$, and so $\dim_{\mathrm{lin}}(\mov(g)) = \dim(\mov(g)) + 1$ and $d(g) = 2$.

It is not clear whether there is a similar formula that simultaneously encompasses these two results and our second main theorem.
\end{remark}

\begin{remark}
In the linear case, Theorem~\ref{thm:linear} shows that an element $w$ of a finite real reflection group $W$ acting on the Euclidean vector space $V = \RR^n$ has the same reflection length in $W$, in any sub-reflection group that contains it, in the orthogonal group $O(V)$, and in the general linear group $\GL(V)$.  We see that this is \emph{not} true in the affine setting.  For example, the glide reflection
\[
\gg : (x, y) \mapsto (x + 1, -y)
\]
has reflection length $3$ as an isometry of $\RR^2$, but it has reflection length $2$ as an element of $\GA(\RR^2)$:
$\gg = \aff{r}_1 \cdot \aff{r}_2$ where $\aff{r}_1 : (x, y) \mapsto (x + y, -y)$ fixes the line $y = 0$ and $\aff{r}_2 : (x, y) \mapsto (x - y + 1, y)$ fixes the line $y = 1$.   
Similarly, the reflection length of an isometry $w$ of $\RR^n$ varies depending on whether one allows all orthogonal reflections or only those belonging to an affine Coxeter group containing $w$, and in the latter case on the choice of which particular Coxeter group $w$ belongs to.
\end{remark}

\subsection{Other work on orthogonal linear groups}

In Theorem~\ref{thm:linear}(b), we slightly mischaracterized the result of Scherk on the orthogonal linear group: he worked in considerably more generality, considering the orthogonal group $O(V, \beta)$ over any finite-dimensional vector space with any nondegenerate symmetric bilinear form $\beta$, provided that the underlying field has characteristic other than $2$.  In this setting, the stated result is valid as long as the form $\beta$ is anisotropic, and this is the case over $\RR$.  Scherk also gives formulas for reflection length in the case that $\beta$ is isotropic.  We are not aware of a corresponding affine result at this level of generality.

Other work on the factorization structure of $O(V)$ for general $V$ includes the case of characteristic $2$ \cite{Dieudonne, Callan}, Wall's parametrization of $O(V)$ in terms of the moved space \cite{Wall1, Wall2} in the general setting, and the close association with the subspace poset in the anisotropic case \cite{BradyWattOrthogonal}.  For a thorough treatment, see \cite[Ch.~11]{Taylor}.

\subsection{Other work on general linear groups}

Theorem~\ref{thm:linear}(c) is not identical to \cite[Thm.~1.1]{Dieudonne}, which is concerned with a slightly more refined question: tracking the number of \emph{transvections} (see Remark~\ref{rmk:transvection} for the definition) needed in a reflection factorization of an element of the general linear group.   As a result, the theorem statement there must be combined with a short additional argument (using the unnumbered remark immediately following the proof of Theorem 1.1 in \cite[p.~152]{Dieudonne}) to give our Theorem~\ref{thm:linear}(c).

It is possible to give a shorter direct proof of Theorem~\ref{thm:linear}(c) if one does not request the additional information about how many transvections (versus semisimple reflections) are used.  Such a proof may be found (in English, and available online) as \cite[Prop.~2.16]{HLR}.  (The authors of that paper were unaware of Dieudonn\'e's prior work.)

\subsection{Affine version of the first main theorem?}

One natural question to ask is whether there is an affine version of Theorem~\ref{thm:linear2} and the first main theorem, that is, whether there is a natural geometric way to tell whether a product of affine reflections is a shortest factorization of its product just by examining the factors.  We are not aware of any such result in any of the three relevant contexts (Coxeter groups, Euclidean isometries, general affine group).

\subsection{Classification of hyperbolic elements over $\FF_2$}

As discussed in Section~\ref{sec:F_2}, the set of hyperbolic elements over $\FF_2$ includes the nontrivial translations, but also additional maps.  These are not difficult to classify: if $\VV$ is an affine space over $\FF_2$, then every non-translation hyperbolic element in $\GA(\VV)$ is a glide reflection, that is, it consists of a reflection followed by a translation.  Indeed, if $(a + \fixlin(\gg)) \cup (\gg(a) + \fixlin(\gg)) = \VV$, we must either have that $\fixlin(\gg) = V$ or that $\fixlin(\gg)$ is a hyperplane in $V$.  In the former case, $\gg$ is a translation.  In the latter case, this implies that the projection $g := \pi(\gg)$ in $\GL(V)$ is a reflection, and so that $\iota(g)$ is a reflection in $\GA(\VV)$.  Any element of $\GA(\VV)$ differs from its image under $\iota\circ \pi$ by a translation (an element in the kernel of $\pi$), which verifies the claim.

\subsection{Coincidences among small groups}

The following group isomorphisms are amusing; as far as we know, they are best explained by the ``law of small numbers" \cite{lawofsmallnumbers}.

The group $\GL_1(\FF_2)$ is the trivial group, fixing the one nonzero vector in $\FF_2$ and isomorphic to the symmetric group $S_1$.  The group $\GA_1(\FF_2)$ is the group with two elements, permuting the two points in $\FF_2$ and isomorphic to the symmetric group $S_2$.

The group $\GL_2(\FF_2)$ is isomorphic to the symmetric group $S_3$, permuting the three lines through the origin in $\FF_2^2$.  In fact, this is an isomorphism as reflection groups, in that the reflections in $\GL_2(\FF_2)$ correspond to the transpositions in $S_3$, and these are exactly the reflections in the standard representation of $S_3$ as a Coxeter group acting in $\RR^2$ (equivalently, as permutation matrices acting on $\RR^3$).

The group $\GA(\FF_2^2)$ is isomorphic as a reflection group to the symmetric group $S_4$, permuting the four points of $\FF_2^2$.  With this identification, the transpositions are again the reflections, the eight three-cycles are the parabolic elements, the three fixed-point-free involutions are the translations, and the six four-cycles are the ``extra" hyperbolic elements.

\section*{Acknowledgements}

The authors thank Jon McCammond for helpful comments, and they thank Vic Reiner for his support, encouragement, and suggestions.  EGD was supported in part by NSF grants DMS-1148634, 1601961.  JBL was supported in part by NSF grant DMS-1401792.

\bibliography{AffineBib}{}
\bibliographystyle{alpha}

\end{document}